\documentclass[12pt]{article}
\usepackage{latexsym}
\usepackage{amsmath,amssymb}
\usepackage{bm}
\usepackage{graphicx,xcolor}
\usepackage{amssymb}
\usepackage{cite}
\usepackage{amsfonts, amscd, amsthm}
\usepackage[all]{xy}
\usepackage{here}
\usepackage{ulem}
\usepackage{mathabx}
\usepackage[titletoc,title]{appendix}
\usepackage{mathbbol}
\usepackage{amsthm}

\theoremstyle{plain}
\theoremstyle{definition}
\newtheorem{thm}{Theorem}[section]
\newtheorem*{thm*}{Theorem}
\newtheorem{example}{Example}[section]
\newtheorem*{example*}{Example}

\newtheorem*{defi*}{Definition}

\renewcommand{\theequation}{\thesection.\arabic{equation}}

\newcommand{\R}{\mathbb{R}}
\newcommand{\Z}{\mathbb{Z}}
\newcommand{\F}{\mathbb{F}^{\times}_p}
\newcommand{\FF}{\mathbb{F}_p}

\newcommand{\md}{{\rm mod}}

\makeatletter
\@addtoreset{equation}{section}
\renewcommand{\theequation}{\thesection.\@arabic\c@equation}
\makeatother

\makeatletter
\renewcommand\appendix{\par
  \setcounter{section}{0}%
  \setcounter{subsection}{0}%
  \gdef\thesection{Appendix \@Alph\c@section }
  \renewcommand{\theequation}
  {\Alph{section}.\arabic{equation}}
}
\makeatother

\makeatletter
\newcounter{subeqncnt}
\def\thesubeqncnt{\alph{subeqncnt}}
\def\subequations{\begingroup%
\stepcounter{equation}\edef\@tempa{\theequation}%
\let\c@equation\c@subeqncnt\c@subeqncnt\z@
\edef\theequation{\@tempa\noexpand\thesubeqncnt}}

\makeatother

\allowdisplaybreaks

\begin{document}

\titlepage


\title{%
A Quadratic Curve Analogue\\ of the Taniyama-Shimura Conjecture
} 
\author{Masahito Hayashi\thanks{masahito.hayashi@oit.ac.jp}\\
Osaka Institute of Technology, Osaka 535-8585, Japan\\
Kazuyasu Shigemoto\thanks{shigemot@tezukayama-u.ac.jp} \\
Tezukayama University, Nara 631-8501, Japan\\
Takuya Tsukioka\thanks{tsukioka@bukkyo-u.ac.jp}\\
Bukkyo University, Kyoto 603-8301, Japan\\
}
\date{\empty}


\maketitle
\abstract{%
For quadratic curves over $\mathbb{F}_p$, the number of solutions, which is governed by an analogue of the Mordell-Weil group, is expressed with the Legendre symbol of a coefficient of quadratic curves.
Focusing on the number of solutions, a quadratic curve analogue of the modular form in the Taniyama-Shimura conjecture is proposed. 
This modular form yields the Gaussian sum and also possesses some modular transformation structure. 

}

\section{Introduction} 
\setcounter{equation}{0}

For exactly solvable models in non-linear systems, the group
structure plays an important role. The KdV equation is one of the examples. 
The AKNS formalism exposes that the KdV equation has the 
$sl(2, \R)\cong so(2,1)/\Z_2 \cong sp(2, \R) \cong su(1,1)$ Lie algebra 
structure \cite{AKNS}. 
From geometrical approaches, we can also observe that the KdV equation has such Lie algebra 
structure\cite{Bianchi}\cite{Hermann}\cite{Crampin}\cite{Sasaki}. 
In addition, we can find the Lie group structure for the Jacobi type elliptic function\cite{Hayashi}. 
It is worth to mention that
the $\wp$-function is one of the static solutions of the KdV equation.
While, in some sense,
the Taniyama-Shimura conjecture\cite{Eichler}\cite{Shimura}\cite{Taylor} is considered to be the 
exactly solvable system because we can obtain all solutions for each elliptic curve over
specific $\FF$ from the Mordell-Weil group structure\cite{Mordell}\cite{Weil}
of that elliptic curve.
Parametrizing elliptic curves by the $\wp$-function, the Lie group structure of the $\wp$-function 
and the Mordell-Weil group in the Taniyama-Shimura conjecture are
strongly connected. The Mordell-Weil group can be considered to be the special Abelian subgroup
of non-Abelian $SL(2, \R)$ Lie group. In the Taniyama-Shimura conjecture, 
the Mordell-Weil group plays an essential role. 

In this paper, we consider an analogue of the Taniyama-Shimura conjecture for
quadratic curves over $\FF$. We first demonstrate to give the 
number of solutions for quadratic curves over $\FF$.
We find that the analogue of the Mordell-Weil group plays an important role in determining
the number of solutions.
Next, by constructing the analogue of the modular form of the Taniyama-Shimura 
conjecture for quadratic curves, we will demonstrate that such an analogue of the modular form 
can be considered as a generalization of the Gaussian sum and it has a structure of the modular transformation by using the other Gaussian sum.

\section{%
Number of Solutions of Quadratic Curves over $\FF$
}
Here we consider the number of solutions of quadratic curves 
over finite fields, which is the quadratic curve analogue of the Taniyama-Shimura conjecture. 
A formula which gives the number of solutions of quadratic curves in finite fields is 
elegantly derived by many authors\cite{Siegel}\cite{Braun}. 
Here we give an elementary and simple proof, which will help to understand the 
quadratic curve analogue of the Taniyama-Shimura conjecture.

In the following, we use the Legendre symbol defined for an odd prime number $p$ :
\[
\left( \frac{\mathstrut a}{p}\right)
=\left\{ \begin{array} {rl}
 1, & \text{if $a$ is a quadratic residue modulo $p$ and $a \not\equiv 0\ (\md\ p)$},\\
-1, & \text{if $a$ is a quadratic nonresidue modulo $p$},\\
 0, & \text{if $a \equiv 0\ (\md\ p)$}.
\end{array} \right.
\]
This definition is equivalent to 
\begin{align}
  \left( \frac{\mathstrut a}{p}\right)
  \equiv a^{\frac{p-1}{2}} \ (\md\ p)
  \quad \mathrm{and}\quad \left( \frac{\mathstrut a}{p}\right) \in \{-1,0,1\}.
\label{eq:Euler_criterion}
\end{align}

\subsection{%
Number of solutions of $\bm{m_1^{} x_1^2+m_2^{} x_2^2 \equiv n \pmod p}$ 
over $\FF$
}

Let us consider the number of solutions $N(p)$ of
\begin{align}
m_1^{} x_1^2+m_2^{} x_2^2 \equiv n \pmod p,
\label{eq:quadratic_curve}
\end{align}
over $\FF$.
For $p=2$, we always obtain two solutions. Then, hereafter, we consider only $p \ne 2$, namely,
$p$ is an odd prime number. 

\begin{thm}
$N(p)$ is given by 
$$\displaystyle{N(p)=p-\left(\frac{-m_1 m_2}{p}\right)}.$$ 
\end{thm}

\begin{proof}
Multiplying $\bar{n}$, which is the inverse of $n$ in $\F$,
to Eq.\eqref{eq:quadratic_curve}, we obtain 
$\bar{n} m_1^{} x_1^2+\bar{n}m_2^{} x_2^2\equiv 1 \pmod p$. To find the number 
of solutions to this equation, we consider the following pair of equations 
$\bar{n} m_1^{} x_1^2\equiv b \pmod p$ and $\bar{n}m_2^{} x_2^2\equiv 1-b \pmod p$.
By using the fact that the number of solutions of $X^2 \equiv a \pmod p$ is given by 
$\left\{\left( \dfrac{a}{p} \right) +1\right\}$, 
the number of solutions of Eq.\eqref{eq:quadratic_curve} is given by
\begin{align}
N(p)
=\sum_{b=0}^{p-1} \left\{ \left(\frac{n m_1 b}{p}\right) +1\right\}
\left\{ \left(\frac{n m_2 (1-b)}{p}\right) +1\right\}
=p+\left(\frac{-m_1 m_2}{p}\right) \sum_{b=0}^{p-1}  \left(\frac{b(b-1)}{p}\right), 
\label{2e3}
\end{align}
where we used Eq.\eqref{eq:Euler_criterion} and
$\displaystyle{\sum_{b=0}^{p-1}  \left(\frac{b}{p}\right)=0}$.

Next, we calculate the same $N(p)$ in another way, that is, we consider
the pair of equations  $m_1^{} x_1^2\equiv c \pmod p$ and 
$m_2^{} x_2^2\equiv n-c \pmod p$. Then we obtain another expression of 
$N(p)$ in the form
\begin{align}
N(p)
=\sum_{c=0}^{p-1} \left( \left(\frac{ m_1 c}{p}\right) +1\right)
\left\{ \left(\frac{m_2 (n-c)}{p}\right) +1\right\}
=p+\left(\frac{-m_1 m_2}{p}\right) \sum_{c=0}^{p-1}  \left(\frac{c(c-n)}{p}\right). 
\label{2e4}
\end{align}
By comparing Eq.(\ref{2e3}) with Eq.(\ref{2e4}), we obtain
an $n$-independent number defined by $S$:
\begin{align}
S
=\sum_{b=0}^{p-1}  \left(\frac{b(b-1)}{p}\right)
=\sum_{b=0}^{p-1}  \left(\frac{b(b-n)}{p}\right). 
\label{2e5}
\end{align}

Next, from the square of the relation 
$\displaystyle{\sum_{b=0}^{p-1} \left(\frac{b}{p}\right)}=0$, we obtain
\begin{align}
0=&\sum_{b=0}^{p-1} \left(\frac{b^2}{p}\right)
  +\sum_{b=0}^{p-1} \left(\frac{b(b-1)}{p}\right)
  +\sum_{b=0}^{p-1} \left(\frac{b(b-2)}{p}\right)   
  +\cdots 
  +\sum_{b=0}^{p-1} \left(\frac{b(b-(p-1))}{p}\right) 
\nonumber\\
 =&(p-1)+S \cdot (p-1)=(p-1)(S+1), 
\label{2e6}
\end{align}
which gives  $S=-1$. Thus we obtain 
\begin{align}
S=\sum_{b=0}^{p-1}  \left(\frac{b(b-1)}{p}\right)
 =\sum_{b=0}^{p-1}  \left(\frac{b(b-n)}{p}\right)=-1, 
\label{2e7}
\end{align}
and 
\begin{align}
N(p)=p+\left(\frac{-m_1 m_2}{p}\right) S=p-\left(\frac{-m_1 m_2}{p}\right) .
\label{2e8}
\end{align}
\end{proof} 
%
Denoting the combination of $p-N(p)$ as $b(p)$, we have  
\begin{align}
b(p)=p-N(p)=\left(\frac{-m_1 m_2}{p}\right).
\label{2e9}
\end{align}

We can transform the general quadratic curves into the form 
$y^2 \equiv a x^2+1$.
We first rewrite Eq.\eqref{eq:quadratic_curve} into 
the form
$(m_2\,x_2)^2= -m_1m_2\,x_1^2+m_2\,n$, and denote
$m_2\,x_2=y$, $x_1=x$, $-m_1 m_2=a$. 
Because the number of solutions is independent of $m_2 n(\ne 0)$,
we can choose $m_2 n=1$. Therefore, hereafter, we adopt
$$
  y^2 \equiv a x^2+1 \pmod p,
$$
as quadratic curves over $\FF$. Here Eq. (\ref{2e8}) turns to be 
\begin{equation}
N(p)=p-\left(\frac{a}{p}\right).
\label{2e10}
\end{equation}

We can calculate $N(p)$ in another way. 

\begin{thm}
The other expression of $N(p)$ is given by 
$$
\displaystyle{N(p)=p+\sum_{x=0}^{p-1} \left(\frac{a x^2+1}{p}\right)}.
$$
\end{thm}

\begin{proof}
The number of solutions that satisfy $y^2 \equiv a x^2+1 \pmod p$ can be found by classifying $x$ into the following three cases, 
$$
\textrm{i}) \   \left(\dfrac{ax^2+1}{p} \right)=+1, \quad  
\textrm{ii}) \  \left(\dfrac{ax^2+1}{p} \right)=-1, \quad \textrm{and}\quad  
\textrm{iii}) \ \left(\dfrac{ax^2+1}{p} \right)= 0. 
$$
For each $x$, there are 2 solutions for i), 0 for ii) and 1 for iii).
Then the number of solutions is given by
\begin{align}
N(p)
  =\sum_{x=0}^{p-1} \left\{ \left(\frac{a x^2+1}{p}\right)+1\right\}
  =p+\sum_{x=0}^{p-1} \left(\frac{a x^2+1}{p}\right).
\label{2e11}
\end{align}
\end{proof} 

Combining two expressions of $N(p)$, we obtain 
\begin{align}
b(p)=p-N(p)=\left( \frac{a}{p}\right)=-\sum_{x=0}^{p-1} \left(\frac{a x^2+1}{p}\right).
\label{2e12}
\end{align}
Our method is applicable to obtain the number of solutions for the elliptic curves
over $\FF$ in the form $y^2 \equiv x^3 + k_2x^2 +k_1x+k_0 \pmod p$\cite{Artin} as 
\begin{align}
b(p)=p-N(p)=-\sum_{x=0}^{p-1} \left(\frac{x^3 + k_2x^2 +k_1x+k_0}{p} \right) .
\label{2e13}
\end{align}

\subsection{%
The quadratic curve analogue of the Mordell-Weil group
}

We give another explanation to Eq.\eqref{2e10} from the point of view of group theory.

\begin{thm}
Let $S\subset\mathbb{F}_p^2$ as 
$$
S=\big\{\, (x, \, y)\ \big| \ y^2\equiv ax^2+1 \pmod p \, \big\}. 
$$
Defining an addition of two points $\textrm{P} (x_1, \, y_1)$ and $\textrm{Q} (x_2, \, y_2)$ as 
\begin{equation}
\textrm{P} (x_1, \, y_1)+\textrm{Q} (x_2, \, y_2)=\textrm{R} (x_3, \, y_3):=(x_1 y_2+y_1 x_2, \, y_1 y_2+a x_1 x_2)\in S, 
\label{2e14}
\end{equation}
$S$ forms a group. 
\end{thm}

\begin{proof}
If we parametrize the quadratic curves $y^2=a x^2+1$ 
by the trigonometric functions
in the form $i \sqrt{a} x_i=\sin\theta_i$, $y_i=\cos\theta_i$,
the above group law of the addition of points on the quadratic curves 
${\rm P}+{\rm Q}={\rm R}$
is determined in such a way that it becomes equivalent to the 
trigonometric addition formula
$$
\sin\theta_3=\sin\theta_1\cos\theta_2+\cos\theta_1\sin\theta_2, \quad
\cos\theta_3=\cos\theta_1\cos\theta_2-\sin\theta_1\sin\theta_2.
$$
This parametrization makes it easy to see an associativity of additions as well as to find a unit element and inverse elements.
\end{proof}

We call $S$ as the quadratic curve analogue of the Mordell-Weil group.
In order to see how $N(p)$ is related to the group structure, 
we first give an example, which will be helpful in understanding 
the proof for the general case.

\begin{example}
$y^2 \equiv -3 x^2+1 \pmod p$.
\[
\begin{tabular}{|c||c|c|c|c|c|c|c|c|c|} \hline
$p$\rule[-2mm]{0mm}{7mm}
&   3  &   5  &  7  &  11  & 13  &  17  & 19  &  23  &  29  \\\hline
$N(p)$\rule[-2mm]{0mm}{7mm}
&   6  &   6  &  6  &  12  & 12  &  18  & 18  &  24  &  30  \\\hline 
$b(p)=p-N(p)$\rule[-2mm]{0mm}{7mm}
& $-3$ & $-1$ & ~1~ & $-1$ & ~1~ & $-1$ & ~1~ & $-1$ & $-1$ \\\hline
$\left(\dfrac{-3}{p}\right)$\rule[-5mm]{0mm}{12mm}
& $-$  & $-1$ & ~1~ & $-1$ & ~1~ & $-1$ & ~1~ & $-1$ & $-1$ \\\hline
\end{tabular}
\]\\
We can see that $b(p)=p-N(p)=\left( \dfrac{-3}{p} \right)$ 
for the prime of $(p,3)=1$. 
For the case of $(p,3) \ne 1$,\break namely $p=3$,
we obtain two solutions $y=1,2$ for each $x=0, 1, 2$.
Then we obtain $N(p)=6$ and $b(p)=p-N(p)=-3$.
This result can be generalized to 
$N(p)$'s and $b(p)$'s of $y^2 \equiv ax^2+1\ (\md\ p)$,
where $a=\pm p,\pm2p,\cdots$. In these cases, $y^2 \equiv 1\ (\md\ p)$. This equation has two solutions $y = 1,\,p-1$.
Since $x$ is arbitrary, $N(p)=2p$ and $b(p)=p-N(p)=-p$.

For $p=5$, solutions are given by
$$(x,y)=(0,1), ~(0,4), ~(2,2), ~(2,3), ~(3,2), ~(3,3),$$
and the number of solutions $N(5)=6$. This ``6'' is the order of the quadratic curve
analogue of the Mordell-Weil group, as we show in the following way.
We start from  ${\rm P}_1=(2,3)$, and use the addition formula
Eq.\eqref{2e14} with $a=-3$ by identifying $x_1=x_2=2$
and $y_1=y_2=3$. Then we have $x_3=12\equiv 2 \pmod 5$,  
$y_3=-3\equiv 2 \pmod 5$. Namely, we obtain 
P$_1$+P$_1$=$[2] {\rm P}_1\equiv (2,2)$.
By similar calculations, we obtain
$$
  [2]{\rm P}_1 \equiv (2,2),\ [3]{\rm P}_1 \equiv (0,4),\ [4] {\rm P}_1 \equiv (3,2),
\ [5]{\rm P}_1 \equiv (3,3),\ [6]{\rm P}_1 \equiv (0,1)={\rm E} \pmod 5.
$$ 
${\rm E}=(0,1)$ is the unit element of the group.

We can understand this group law by rewriting the quadratic curves 
$3 x^2+y^2=1 \pmod p$ to $X^2+Y^2=1 \pmod p$ 
with $X=\sqrt{3}x$, $Y=y$
by using the algebraic integer $\sqrt{3}$. 
$(X,Y)$ is an element of $SO(2)$.
We further rewrite it to 
$Z=i X+Y=\sqrt{-3} x+y, \pmod p$ with $Z^6=1 \pmod p$.
$Z$ is an element of $U(1)$.
Then we associate $(x_i,y_i)$ and $Z_i$ as shown in the table below: 
\[
\begin{tabular}{|c||c|c|c|c|c|c|} \hline
$(x_i,y_i)$\rule[-2mm]{0mm}{7mm}&
(0,1)&(0,4)&(2,2)&(2,3)&(3,2)&(3,3)  \\\hline
$Z_i=\sqrt{\mathstrut -3}\,x_i+y_i$\rule[-2mm]{0mm}{7mm}&
$Z_1$&$Z_2$&$Z_3$&$Z_4$&$Z_5$&$Z_6$  \\\hline
\end{tabular}
\]
If we take $Z_4$ as the generator, we have 
$$
(Z_4)^2 \equiv Z_3,\
(Z_4)^3 \equiv Z_2,\ 
(Z_4)^4 \equiv Z_5,\ 
(Z_4)^5 \equiv Z_6,\ 
(Z_4)^6 \equiv Z_1, \pmod 5,
$$ 
%
where $Z_1$ is the unit element.
This form is easier to understand that the quadratic curve analogue of 
the Mordell-Weil group is the cyclic group.

For the quadratic curve, the solution at infinity does not constitute
the element of this quadratic curve analogue of Mordell-Weil group. 
From this reason, we can choose another generator such as ${\rm Q}_1=(3,3)$, and we have 
$$
[2] {\rm Q}_1\equiv (3,2),\ 
[3] {\rm Q}_1\equiv (0,4),\
[4] {\rm Q}_1\equiv (2,2),\
[5] {\rm Q}_1\equiv (2,3),\  
[6] {\rm Q}_1\equiv (0,1)={\rm E} \pmod 5.$$
\end{example}

\begin{thm}
The quadratic curve analogue of 
the Mordell-Weil group for the curve $y^2 \equiv a x^2+1 \pmod p$ becomes the cyclic group and the 
number of solutions $N(p)$ is the order of this cyclic group.
\end{thm}
 
\begin{proof}
If $\left(\dfrac{a}{p} \right)=1$, we set $a=r^2$ where $r$ is the proper primitive root of $\F$.
Then we first rewrite the quadratic curves to
$X^2+Y^2 \equiv 1\ ({\rm mod}\ p)$ with 
$X=\sqrt{-r^2}x=i r x$, $Y=y$. Further, we make the 
combination in the form  $Z=-i X+Y=r x+y$.
Then solutions of the quadratic curves $(x,y)$ is rewritten 
with $Z$ in the form  $Z= r x+y$, then $Z^p \equiv r^p x^p+y^p\equiv 
r x+y \equiv Z \pmod p$. From the property of the primitive 
root, $Z, Z^2, \cdots, $ and  $Z^{p-1}$ are all different
since the smallest integer $n$ which satisfies 
$Z^n \equiv 1 \pmod p$ is $p-1$. That is, the number of the solution is 
$N(p)=p-1=p-\left(\dfrac{a}{p} \right)$ and 
the quadratic curve analogue of the Mordell-Weil group becomes the cyclic group.

If $\left(\dfrac{a}{p} \right)=-1$, we set $a=r$ where $r$ is the proper primitive root of $\F$.
Then we first rewrite the elliptic curve to
$X^2+Y^2 \equiv 1\ ({\rm mod}\ p)$ with $X=\sqrt{-r}x=i \sqrt{r} x$, $Y=y$. 
Further we put $Z=-i X+Y=\sqrt{r} x+y$.
Thus solutions of the quadratic curves $(x,y)$ is rewritten 
with $Z$ in the form  $Z=\sqrt{r} x+y$. Then $Z^p \equiv r^{p/2}x^p+y^p\equiv 
-\sqrt{r} x+y \equiv \bar{Z} \pmod p$, $Z^{p+1} \equiv Z \bar{Z} \equiv 1  \pmod p$, where we used 
$r^{\frac{p-1}{2}} \equiv -1 \pmod p$ for $\left(\dfrac{r}{p} \right)=-1$
yielded from Eq.\eqref{eq:Euler_criterion}. 
From the property of the primitive 
root, $Z, Z^2, \cdots, $ and $Z^{p+1}$ are all different
since the smallest integer $n$ which satisfies 
$Z^n \equiv 1 \pmod p$ is $p+1=1-\left(\dfrac{a}{p} \right)=N(p)$. 
\end{proof}

\section{%
Generalized Gaussian Sum}
Here we consider the quadratic curve analogue of the modular form of the Taniyama-Shimura conjecture.

\subsection{%
The quadratic curve anologue of the modular form of the 
Taniyama-Shimura conjecture}

We denote the number of solutions for the elliptic curve 
$$
y^2 \equiv x^3+k_2 x^2+k_1 x+k_0 \pmod p,
$$
over $\FF$ by $\hat{N} (p)$ and define $\hat{b}(p)=p-\hat{N}(p)$. Suppose a modular form, which corresponds to this elliptic curve, to be 
$$
 f(\tau)=\sum_{n=1}^{\infty} \hat{c}(n) q^n,  \quad q=\exp{(2 \pi i \tau)}.
$$
Then the Taniyama-Shimura conjecture claims $\hat{b}(p)=\,\hat{c}(p)$ for prime numbers $p$.
In our quadratic case, i.e.
$$
y^2 \equiv a x^2+1\ (\md\ p), 
$$
over $\FF$, the number of solutions is given by 
$N(p)=p-\left( \dfrac{a}{p} \right)$ and we define $b(p)=p-N(p)=\left( \dfrac{a}{p} \right)$. 
Then the quadratic curve analogue of the modular form is given by
$$
f(\tau)=\sum_{n=1}^{\infty} c(n) q^n,  \quad q=\exp{(2 \pi i \tau) }.
$$
with $c(p)=b(p)=\left( \dfrac{a}{p} \right)$  for prime numbers $p$. 
 
Based on the above discussions, the quadratic curve analogue 
of the modular form should at least include
$\left( \dfrac{\mathstrut a}{p}\right) q^p$ terms.  
For such a form to have some modular transformation property, it must include terms $b(n) q^n$ with a non-prime integer $n$. Candidate that satisfies these requirements can be obtained by replacing the prime number $p$ with any integer $n$. At the same time, we have to replace the Legendre symbol with the Kronecker symbol in the form
$\left( \dfrac{\mathstrut a}{p}\right) q^p \rightarrow \left( \dfrac{\mathstrut a}{n}\right)_K q^n\rule[-12pt]{0pt}{30pt}$,
because the Legendre symbol is not defined for non-prime integer $n$. Note that for odd prime number $p$,
$
 \left( \dfrac{\mathstrut a}{p}\right)_K
=\left( \dfrac{\mathstrut a}{p}\right)
\rule[-12pt]{0pt}{30pt}$.

Let us discuss the periodicities of the Kronecker symbol. For $a \equiv 1 \pmod 4$, we obtain 
$$
 \left(\frac{\mathstrut a}{p}\right)_K
=\left(\frac{\mathstrut p}{a}\right)_K (-1)^{\frac{a-1}{2} \cdot \frac{p-1}{2}}
=\left(\frac{\mathstrut p+a\ell}{a}\right)_K (-1)^{\frac{a-1}{2} \cdot \frac{p+a\ell-1}{2}}
=\left(\frac{\mathstrut a}{p+a \ell} \right)_K,
$$ 
for any integer $\ell$. That is, $\left(\dfrac{\mathstrut a}{p}\right)_K$ is periodic with respect to $p \pmod a$.

For $a \equiv 3 \pmod 4$, we obtain 
$$
 \left(\frac{\mathstrut a}{p}\right)_K
=\left(\frac{\mathstrut p}{a}\right)_K (-1)^{\frac{a-1}{2} \cdot \frac{p-1}{2}}
=\left(\frac{\mathstrut p+4a \ell}{a}\right)_K (-1)^{\frac{a-1}{2} \cdot \frac{p+4a\ell-1}{2}}
=\left(\frac{\mathstrut a}{p+4a\ell}\right)_K.
$$ 
That is, $\left(\dfrac{\mathstrut a}{p}\right)_K$ is periodic with respect to $p \pmod{4a}$.

For $a \equiv 2 \pmod 4$, we put $a=2 a'$ with $a'$ is odd integer. Then we have 
\begin{align*}
  \left(\frac{\mathstrut a}{p}\right)_K
&=\left(\frac{\mathstrut 2}{p}\right)_K \left(\frac{\mathstrut a'}{p}\right)_K
 =(-1)^{\frac{\mathstrut p^2-1}{8}}
 \left(\frac{\mathstrut p}{a'}\right)_K (-1)^{\frac{a'-1}{2} \cdot \frac{p-1}{2}} \\
&=(-1)^{\frac{\mathstrut (p+8 a' \ell)^2-1}{8}}\left(\frac{p+8 a' \ell}{a'}\right)_K 
 (-1)^{\frac{\mathstrut a'-1}{2} \cdot \frac{p+8 a' \ell-1}{2}}
=\left(\frac{\mathstrut a}{p+4a\ell}\right)_K.
\end{align*}
That is, $\left(\dfrac{\mathstrut a}{p}\right)_K$ is periodic with respect to $p \pmod{4a}$.

Then we arrive at the quadratic curve analogue of the modular form 
\begin{align}
f(\tau)=
\left\{ \begin{array}{ll}\displaystyle
~\sum_{n=1}^{\infty}~
\left( \frac{\mathstrut a}{n}\right)_K q^n,\quad q=\exp( 2 \pi i \tau), 
& \text{if~}a \equiv 1 \pmod 4,
\\[10pt]\displaystyle
\sum_{\genfrac{}{}{0pt}{}{n=1}{(n, 4a) = 1}}^{\infty}\!\!\!
\left( \frac{\mathstrut a}{n}\right)_K q_1^n, \quad q_1=\exp( 2 \pi i \tau/4), 
& \text{if~} a \equiv 2, 3 \pmod 4.
\end{array} \right.
\label{3e1}
\end{align}
In order to concrete our claim, we will show the followings in subsequent subsections:
\begin{enumerate}
  \item $f(\tau)$ becomes the Gaussian sum if $\tau$ is the special value $\tau_0$,
  \item $f(\tau)$ is associated with the theta function, so it has the structure of the modular transformation. 
\end{enumerate}
Then we call the above infinite sum Eq.\eqref{3e1} the generalized Gaussian sum. 
If $\tau$ is the special value $\tau_0$, the generalized Gaussian sum becomes periodic, and we denote
one of its periods, which is proportional to the Gaussian sum, as $\bar{f}(\tau_0)$.

We close this subsection by giving some examples of $\bar{f}(\tau_0)$, which will be 
helpful to understand the proof in the following subsections. 

\begin{example}
${}$

\begin{enumerate}
  \item For $y^2 \equiv~~5 x^2+1 \pmod p$,\ 
we obtain $\bar{f}(1/5)=q-q^2-q^3+q^4=\sqrt{5}=G_5$.
  \item For $y^2 \equiv -3 x^2+1 \pmod p$,\ 
we obtain $\bar{f}(1/3)=q-q^2=\sqrt{-3}=G_3$.
  \item For $y^2 \equiv~~3 x^2+1 \pmod p$,\ 
we obtain $\bar{f}(1/3)=q_1-q_1^5=\sqrt{3}=-iG_3$.
  \item For $y^2 \equiv -5 x^2+1 \pmod p$,\ 
we obtain $\bar{f}(1/5)=q_1+q_1^3+q_1^7+q_1^9=\sqrt{-5}=iG_5$
\end{enumerate}
\end{example}

\subsection{Gaussian sum} 

In this subsection, we show that the Gaussian sum can be extracted from the 
quadratic curve analogue of the modular form givn in Eq.\eqref{3e1}. 

Here we list properties of the Kronecker symbol used in the following calculations.
Let $m=2^{e_1}m', n=2^{e_2}n'$\ ($m', n'=$ odd integer): 
\begin{align}
{\rm i)}\ & \text{if $m>0$ or $n>0$ and $(m,n)=1$
($e_1=0$ or $e_2=0$ and $(m',n')=1$),} 
\nonumber\\
& \left( \frac{\mathstrut n}{m}\right)_K \left( \frac{\mathstrut m}{n}\right)_K
=(-1)^{\frac{m'-1}{2} \cdot \frac{n'-1}{2}}. 
\label{3e2}\\
{\rm ii)}\ &\left( \frac{\mathstrut n}{2}\right)_K= \left( \frac{2}{n}\right)_K
=\left\{ \begin{array} {rl}
          1, & \text{if~}n \equiv 1, 7 \pmod 8,\\
         -1, & \text{if~}n \equiv 3, 5 \pmod 8,\\
          0, & \text{if~}2|n. 
         \end{array}
 \right.
\label{3e3}\\
{\rm iii)}\ &\left( \frac{-1}{n}\right)_K= (-1)^{\frac{n'-1}{2}}. 
\label{3e4}\\
{\rm iv)}\ & \text{if $n \ne -1$ }, 
\nonumber\\
& \left( \frac{a b}{n}\right)_K=\left( \frac{\mathstrut a}{n}\right)_K \left( \frac{b}{n}\right)_K, \quad
\left( \frac{\mathstrut n}{a b}\right)_K=\left( \frac{\mathstrut n}{a}\right)_K \left( \frac{\mathstrut n}{b} \right)_K. 
\label{3e5}\\
\rm v)\ &
\left( \frac{\mathstrut m}{n}\right)_K=\pm1\text{, if $(m,n) = 1$}, \text{ otherwise }\left( \frac{\mathstrut m}{n}\right)_K=0. 
\label{3e6}\\
{\rm vi)}\ & \text{for $n>0$ and}\ a \equiv b \mod{}
\left\{ \begin{array}{rl} 
         4n, & \text{if~}n \equiv 2 \pmod 4,\\
          n, & \text{otherwise~} ,\ 
        \end{array}
\right. 
, \quad \left( \frac{\mathstrut a}{n}\right)_K= \left( \frac{b}{n} \right)_K.
\label{3e7}
\end{align}
%

\subsubsection{$\bm{y^2 \equiv a x^2+1\ (\md\ p)}$,\ $\bm{a>0}$ and 
$\bm{a\equiv 1\ ({\rm mod}\ 4)}$}

When $(n,a)=1$, 
$\left( \dfrac{\mathstrut a}{n}\right)_K=\left( \dfrac{\mathstrut n}{a}\right)_K$ is derived from Eq.\eqref{3e2}.
Since $\left( \dfrac{\mathstrut a}{n}\right)_K=0$ if $(n,a)\neq1$, Eq.\eqref{3e1} can be rewritten as 
\begin{align}
f(\tau)
=\sum_{n=1}^{\infty} \left( \frac{\mathstrut n}{a}\right)_K q^n,\quad  q=\exp(2 \pi i \tau).
\label{3e8}
\end{align}

Eq.\eqref{3e7} shows that 
$\left( \dfrac{\mathstrut n+a}{a}\right)_K=\left( \dfrac{\mathstrut n}{a}\right)_K$ holds.
If $\tau=1/a$, then $q^a=1$, and we obtain
\[
  \left( \frac{\mathstrut n+a}{a}\right)_Kq^{n+a}=\left( \frac{\mathstrut n}{a}\right)_K q^n.
\]
From this equation, it can be seen that $f(1/a)$ repeatedly contains $\bar{f}(1/a)$ shown in the following
\begin{align}
\bar{f}(1/a)=\sum_{n=1}^{a-1} \left( \frac{\mathstrut n}{a}\right)_K \exp(2 \pi i n/a).
\label{3e9}
\end{align}
Since $\left(\dfrac{\mathstrut a}{a}\right)_K=0$, the sum of $n$ in Eq.\eqref{3e9} is up to $a-1$.

If we put $a=p\ ({\rm prime\ number})$ further, $\bar{f}(1/p)$ becomes the Gaussian sum in the form 
\begin{align}
 \bar{f}(1/p)
  =
  \sum_{n=1}^{p-1} \left( \frac{n}{p}\right) \exp( 2 \pi i n/p)
  =
  G_p=\sqrt{p},\quad p \equiv 1 \pmod 4. 
\label{3e10}
\end{align} 

\subsubsection{$\bm{y^2 \equiv a x^2+1 \pmod p}$,\ $\bm{a<0}$ and 
$\bm{a\equiv 1 \pmod 4}$}

Let us consider for $a\equiv 1 \pmod 4$ and $a<0$. In this case we rewrite $\left(\dfrac{\mathstrut  a }{n}\right)_K$ as
\[
  \left(\frac{\mathstrut  a }{n}\right)_K
 =\left(\frac{\mathstrut-|a|}{n}\right)_K
 =\left(\frac{\mathstrut -1 }{n}\right)_K
  \left(\frac{\mathstrut |a|}{n}\right)_K
 =(-1)^{\frac{n'-1}{2}}\times
  (-1)^{\frac{|a|-1}{2}\cdot\frac{n'-1}{2}}
  \left(\frac{n}{\mathstrut |a|}\right)_K
\]
by using Eqs.(\ref{3e5}), (\ref{3e4}) and \eqref{3e2}.
Note that $(-1)^{\frac{|a|-1}{2}}=-1$ because of $|a| \equiv 3 \pmod 4$. Then we obtain
\[
  (-1)^{\frac{n'-1}{2}}\times
  (-1)^{\frac{|a|-1}{2}\cdot\frac{n'-1}{2}}
  =
  (-1)^{\frac{n'-1}{2}}\times
  (-1)^{\frac{n'-1}{2}}
  =
  1 .
\]
Then Eq.\eqref{3e1} can be rewritten as 
\begin{align}
 f(\tau)
  =
  \sum_{n=1}^{\infty} \left( \frac{n}{|a|}\right)_K q^n,\quad q
  =
  \exp(2 \pi i \tau).
\label{3e11}
\end{align}
If $\tau=1/|a|$, then $q^{|a|}=1$, and we obtain
\[
  \left( \frac{n+|a|}{|a|}\right)_Kq^{n+|a|}=\left( \frac{\mathstrut n}{|a|}\right)_K q^n.
\]
The generalized Gaussian sum
$f(1/|a|)$ repeatedly contains $\bar{f}(1/|a|)$ shown in the following
\begin{align}
\bar{f}(1/|a|)=\sum_{n=1}^{|a|-1} \left( \frac{n}{|a|}\right)_K \exp(2 \pi i n/|a|).
\label{3e12}
\end{align}

If we put $|a|=p\ ({\rm prime\ number})$ further, $\bar{f}(1/p)$ becomes the Gaussian sum in the form 
\begin{align}
 \bar{f}(1/p)
  =
  \sum_{n=1}^{p-1} \left( \frac{n}{p}\right) \exp( 2 \pi i n/p)
  =
  G_p=i \sqrt{p}, \quad p \equiv \ 3 \pmod 4. 
\label{3e13}
\end{align} 

\subsubsection{$\bm{y^2 \equiv a x^2+1 \pmod p}$,\ $\bm{a>0}$ and 
$\bm{a\equiv 3 \pmod 4}$}

\begin{thm}
Let $a\equiv 3 \pmod 4$ and $a>0$.
$$
  \left( \dfrac{\mathstrut a}{n+2a}\right)_K=-\left( \dfrac{\mathstrut a}{n}\right)_K.
$$ 
\end{thm}

\begin{proof}
Because $a=4k+3$ and $n$ is odd integer from $(n, 4a)=1$, we obtain 
$$
  \left( \frac{\mathstrut a}{n}\right)_K
 =
  (-1)^{\frac{a-1}{2} \cdot \frac{n-1}{2} } 
  \left( \frac{\mathstrut n}{a}\right)_K 
 =
  (-1)^{\frac{n-1}{2} }
  \left( \frac{\mathstrut n}{a}\right)_K , 
$$
from Eq.\eqref{3e2}.
While we have
$$
  \left( \frac{a}{n+2a}\right)_K
 =
  (-1)^{\frac{a-1}{2} \cdot \frac{n+2 a-1}{2} } 
  \left( \frac{n+2a}{a}\right)_K 
 =
  (-1)^{\frac{n-1}{2}+a}
  \left( \frac{\mathstrut n}{a}\right)_K ,
$$
where we used Eq.(\ref{3e7}).
Then we obtain $\left( \dfrac{a}{n+2 a }\right)_K\left/\left( \dfrac{\mathstrut a}{n}\right)_K\right.
=(-1)^{a} =-1$. 
\end{proof}

If $\tau=1/a$, then $q_1^{2a}=-1$, we obtain
\[
  \left( \frac{a}{n+2a}\right)_Kq_1^{n+2a}=\left( \frac{\mathstrut a}{n}\right)_K q_1^n.
\]
The generalized Gaussian sum
$f(1/a)$ repeatedly contains $\bar{f}(1/a)$ shown in the following
\begin{align}
\bar{f}(1/a)
  =
   \sum_{\genfrac{}{}{0pt}{}{n=1}{(n,4a)=1}}^{2a-1}\!\!
   \left( \frac{\mathstrut a}{n}\right)_K q_1^n
  =
   \sum_{\genfrac{}{}{0pt}{}{n=1}{(n,4a)=1}}^{2a-1}\!\!
  (-1)^{\frac{n-1}{2} }
  \left( \frac{\mathstrut n}{a}\right)_K q_1^n.
\label{3e14}
\end{align}

In Eq.\eqref{3e14}, $n$ takes $(a-1)$ values as {$n\in\left\{ 1, 3, \cdots, \widecheck{a},\cdots, 2a-1\right\}$,
where $\widecheck{a}$ indicates that $a$ is not included. These values can be mapped to the following values $m$ as
{$m\in\left\{ 1, 2, 3, \cdots, a-1 \right\}$. The relation between $n$ and $m$ is given with suitable integer $\ell$ as follows
\begin{align}
  n=4m-a(2\ell-1).
\label{eq:mnrelation}
\end{align}
We show here one example for $a=11$. Proof for general $a$ will be given in Appendix A.
\[
\begin{tabular}{|c||c|c|c|c|c|c|c|c|c|c|} \hline
$~m$   \rule[-2mm]{0mm}{7mm}&  1 & 2 &\,3\:&\,4\:&\,5\:& 6 & 7 & 8 &\,9\:&10  \\\hline
$ \ell$\rule[-2mm]{0mm}{7mm}&  0 & 0 &  1  &  1  &  1  & 1 & 1 & 1 &  2  & 2  \\\hline
$~n$   \rule[-2mm]{0mm}{7mm}& 15 &19 &  1  &  5  &  9  &13 &17 &21 &  3  & 7  \\\hline
\end{tabular}
\]
Replace $n$ in Eq.\eqref{3e14} with $m$ to get the following equation
\begin{align}
\bar{f}(1/a)
&=
\sum_{\genfrac{}{}{0pt}{}{n=1}{(n,4a)=1}}^{2a-1}\!\!
\left( \frac{\mathstrut n}{a}\right)_K
 (-1)^{\frac{n-1}{2}} q_1^n
=\sum_{m=1}^{a-1} \left( \frac{4m-a(2\ell-1)}{a}\right)_K
 (-1)^{\frac{4m-a(2\ell-1)-1}{2}} q_1^{4m-a(2\ell-1)}.
\label{3e16}
\end{align}
By using Eqs.\eqref{3e7}, \eqref{3e5} and \eqref{3e3}, we obtain
\[
  \left( \frac{4m-a(2\ell-1)}{a}\right)_K
  =
  \left( \frac{4 m}{a}\right)_K
  =
  \left( \frac{2}{a}\right)_K\left( \frac{2}{a}\right)_K\left( \frac{\mathstrut m}{a}\right)_K
  =
  \left( \frac{\mathstrut m}{a}\right)_K.
\]
Furthermore, we obtain
\[
  (-1)^{\frac{4m-a(2\ell-1)-1}{2}}
  =
  (-1)^{2m-a\ell+\frac{a-1}{2}}
  =
  \left((-1)^{-a}\right)^{\ell}(-1)^{\frac{a-1}{2}}
  =
  (-1)^{\ell+1},
\]
\[
  q_1^{4m-a(2\ell-1)}
  =
  \exp(2\pi im/a)\exp\bigl((-2\ell+1)\pi i/2\bigr)
  =
  i(-1)^{\ell}\exp(2\pi im/a),
\]
where we used $q_1=\exp(\pi i/2a)$.

If we put $a=p\ ({\rm prime\ number})$ further, $\bar{f}(1/p)$ becomes the Gaussian sum in the form 
\begin{align}
 \bar{f}(1/p)
 =
 -i \sum_{m=1}^{p-1} \left( \frac{m}{p}\right) \exp( 2 \pi i m/p)
 =
 -iG_p=\sqrt{p}, 
\quad p \equiv 3 \pmod 4. \label{3e17}
\end{align} 

\subsubsection{$\bm{y^2 \equiv a x^2+1 \pmod p}$,\ $\bm{a<0}$ and 
$\bm{a\equiv 3 \pmod 4}$}

We rewrite $f(\tau)$ as
\begin{align}
f(\tau)
=
\sum_{\genfrac{}{}{0pt}{}{n=1}{(n,4a)=1}}^{\infty}\!\!
\left( \frac{-|a|}{n}\right)_K q_1^n, \quad  
q_1=\exp(2 \pi i \tau/4).
\label{3e18}
\end{align}

\begin{thm}
Let $a\equiv 3 \pmod 4$ and $a<0$.
$$
  \left( \dfrac{-|a|}{n+2|a|}\right)_K=-\left( \dfrac{-|a|}{n}\right)_K.
$$ 
\end{thm}

\begin{proof}
Because $|a|=1 \pmod 4$ and $n=$ odd number from $(n, 4|a|)=1$, 
$(-1)^{\frac{|a|-1}{2}}=1$ and $\frac{n-1}{2}$ is integer. Then, we obtain
$$
  \left( \frac{-|a|}{n}\right)_K
  =
  \left( \frac{-1}{n}\right)_K \left( \frac{|a|}{n}\right)_K 
  =
  (-1)^{\frac{n-1}{2}}
  \times
  (-1)^{\frac{|a|-1}{2} \cdot \frac{n-1}{2} }
  \left( \frac{n}{|a|}\right)_K 
  =
  (-1)^{\frac{n-1}{2}} \left( \frac{n}{|a|}\right)_K.
$$ 
While we obtain
\begin{align}
\left( \frac{-|a|}{n+2|a| }\right)_K
 &=
 \left( \frac{-1}{n+2|a| }\right)_K \left( \frac{|a|}{n+2|a| }\right)_K
 =
 (-1)^{\frac{n+2 |a|-1}{2}} 
 \times
 (-1)^{\frac{|a|-1}{2} \cdot \frac{n+2|a|-1}{2} }
 \left( \frac{n+2| a| }{|a|}\right)_K 
\nonumber\\
&= (-1)^{\frac{n+1}{2}} 
\left( \frac{n}{|a|} \right)_K .
\nonumber
\end{align}
Note that $\dfrac{n+2 |a|-1}{2}=\dfrac{n+1}{2}+|a|-1$. Then we obtain 
$\displaystyle{\left( \frac{-|a|}{n+2| a| }\right)_K\left/\left( \frac{-|a|}{n}\right)_K\right.=-1}$.
\end{proof}

If $\tau=1/|a|$, then $q_1^{2|a|}=-1$, and we obtain
\[
  \left( \frac{-|a|}{n+2|a|}\right)_Kq_1^{n+2|a|}=\left( \frac{\mathstrut -|a|}{n}\right)_K q_1^n.
\]
The generalized Gaussian sum $f(1/|a|)$ repeatedly contains $\bar{f}(1/|a|)$ shown in the following
\begin{align}
\bar{f}(1/|a|)
=
\sum_{\genfrac{}{}{0pt}{}{n=1}{(n,4|a|)=1}}^{2|a|-1}\!\!
\left( \frac{-|a|}{n}\right)_K q_1^n
=
\sum_{\genfrac{}{}{0pt}{}{n=1}{(n,4|a|)=1}}^{2|a|-1}\!\!
(-1)^{\frac{n-1}{2}}
\left( \frac{n}{|a|}\right)_K q_1^n.
\label{3e19}
\end{align}
The relation between
$n\in\left\{ 1, 3, \cdots, |\widecheck{a}|,\cdots, 2|a|-1\right\}$
and
$m\in\left\{ 1, 2, 3, \cdots, |a|-1 \right\}$
is given with suitable $\ell$ as follows
$$
  n=4m-|a|(2\ell-1).
$$
This will be proved in Appendix A.

By the same calculation as shown below Eq.\eqref{3e16}, we have 
\begin{align}
\bar{f}(1/|a|)
&=
\sum_{\genfrac{}{}{0pt}{}{n=1}{(n,4|a|)=1}}^{2|a|-1}\!\!(-1)^{\frac{n-1}{2}} 
\left( \frac{n}{|a|}\right)_K q_1^n\nonumber\\
&=
\sum_{m=1}^{|a|-1}(-1)^{\frac{4m-|a|(2\ell-1)-1}{2}} 
\left( \frac{4m-|a|(2\ell-1)}{|a|}\right)_K q_1^{4m-|a|(2\ell-1)}
\nonumber\\
&=
\sum_{m=1}^{|a|-1}  (-1)^{\ell} 
\left( \frac{m}{|a|}\right)_K \exp(2\pi i m/|a|) \times i (-1)^{\ell}
=i \sum_{m=1}^{|a|-1} \left( \frac{m}{|a|}\right)_K \exp( 2 \pi i m/|a|),
\label{3e20}
\end{align}
where we used $q_1=\exp(\pi i/2|a|)$.

If we put $|a|=p\ ({\rm prime\ number})$ further, $\bar{f}(1/p)$} becomes the Gaussian sum in the form 
\begin{align}
 \bar{f}(1/p)
  =
  i \sum_{m=1}^{p-1} \left( \frac{m}{p}\right) \exp( 2 \pi i m/p)
  =
  iG_p=i \sqrt{p}, \quad p \equiv 1 \pmod 4. 
\label{3e21}
\end{align} 

\subsubsection{$\bm{y^2 \equiv a x^2+1 \pmod p}$,\ $\bm{a\equiv 2 \pmod 4}$ }

We put $a=4k+2=2(2k+1)=2 a'\text{, where $a'$ is odd}$.
Here $n$ is odd owing to $(n, 4a) =1$.

\begin{thm}
Let $a\equiv 2 \pmod 4$ and $a>0$.
$$
  \left( \dfrac{a}{n+2 a} \right)_K=-\left( \dfrac{\mathstrut a}{n} \right)_K.
$$ 
\end{thm}

\begin{proof}
Eq.\eqref{3e3} can be rewritten for odd $n$ as
$
 \left( \dfrac{\mathstrut n}{2}\right)_K
 =
  \left( \dfrac{\mathstrut2}{n}\right)_K
 =(-1)^\frac{n^2-1}{8}.
$
Then we obtain
\begin{align}
&\left( \frac{a}{n+2 a} \right)_K
=\left( \frac{2a'}{n+4 a'} \right)_K
=\left( \frac{2}{n+4 a'}\right)_K \left( \frac{ a'}{n+4 a'} \right)_K
\nonumber\\
&=(-1)^{\frac{(n+4 a')^2-1}{8}} \times
  (-1)^{\frac{a'-1}{2} \cdot \frac{(n+4 a')-1}{2}} 
  \left( \frac{n+4 a'}{a'} \right)_K 
=(-1)^{\frac{n^2-1}{8}} (-1)^{n a'}  \times
  (-1)^{\frac{a'-1}{2} \cdot \frac{n-1}{2}}
  \left( \frac{\mathstrut n}{a'} \right)_K
\nonumber 
\\
&=-(-1)^{\frac{n^2-1}{8}} \times
  (-1)^{\frac{a'-1}{2} \cdot \frac{n-1}{2}} 
  \left( \frac{\mathstrut n}{a'} \right)_K.
\nonumber
\end{align}
While we obtain
\begin{align}
&\left( \frac{\mathstrut a}{n} \right)_K=\left( \frac{2a'}{n} \right)_K
=\left( \frac{2}{n} \right)_K \left( \frac{a'}{n} \right)_K
=(-1)^{\frac{n^2-1}{8}} \times
 (-1)^{\frac{a'-1}{2} \cdot \frac{n-1}{2}}
 \left( \frac{\mathstrut n}{a'} \right)_K  .
\nonumber
\end{align}
Thus we obtain $\left( \dfrac{a}{n+2 a} \right)_K\left/\left( \dfrac{\mathstrut a}{n} \right)_K\right.=-1$.
\end{proof}

If $\tau=1/a$, then $q_1^{2a}=-1$, and we obtain
\[
  \left( \frac{a}{n+2a}\right)_K q^{n+2a}=\left( \frac{\mathstrut a}{n}\right)_K q^n.
\]
The generalized Gaussian sum $f(1/a)$ repeatedly contains $\bar{f}(1/a)$ shown in the following 
\begin{align}
\bar{f}(1/a)
=
\sum_{\genfrac{}{}{0pt}{}{n=1}{(n,4a)=1}}^{2a-1}\!\!
\left( \frac{\mathstrut a}{n} \right)_K \exp( 2 \pi i n/4a).
\label{3e22}
\end{align}
In the case of $a<0$, the similar calculation shows that
\begin{align}
\bar{f}(1/|a|)
=
\sum_{\genfrac{}{}{0pt}{}{n=1}{(n,4|a|)=1}}^{2|a|-1}\!\!
\left( \frac{-|a|}{n} \right)_K \exp( 2 \pi i n/4|a|).
\label{3e23}
\end{align}
 
By using the other expression of Gaussian sum, 
Eqs.\eqref{3e22} and \eqref{3e23} are expected to be rewritten as follows
\begin{align*}
  \bar{f}(1/a)
 &=\frac{1}{1+i}\sum_{n=0}^{2 a-1} \exp( 2 \pi i n^2/4 a)=\sqrt{a},
  \quad\text{if $a>0$ and $a\equiv2 \pmod 4$},\\
  \bar{f}(1/|a|)
 &=\frac{i}{1+i}\sum_{n=0}^{2 |a|-1} \exp( 2 \pi i n^2/4 |a|)=i \sqrt{|a|},
  \quad\text{if $a<0$ and $a\equiv2 \pmod 4$}.
\end{align*}
For $a=\pm 2, \pm 6, \pm 10$, we have verified that our expectation is correct. Below, we show $a=6$ case.
With $q_1=\exp(\pi i /12)$, we obtain
\begin{align*}
  &\sum_{\genfrac{}{}{0pt}{}{n=1}{(n,24)=1}}^{11}\!\!
  \left( \frac{\mathstrut 6}{n} \right)_K q_1^n
  =
  q_1 + q_1^5 - q_1^7 - q_1^{11} = \sqrt{6},
\end{align*}
and 
\begin{align*}
  \sum_{n=0}^{11} q_1^{n^2}
  &=
    1 + q_1^{ } + q_1^{4} + q_1^{9} - q_1^{4} + q_1^{ }
  - 1 + q_1^{ } - q_1^{4} + q_1^{9} + q_1^{4} + q_1^{ } \\
  &=
   4q_1 + 2q_1^9 
   =
   (1+i)\sqrt{6}.
\end{align*}

\subsection{%
Modular transformation structure
}

Here we explain the generalized Gaussian sum is associated with a theta function.

\begin{thm}
There are two expressions of the Gaussian sum in the form
\begin{align}
G_p
=\sum_{n=1}^{p-1} \left(\frac{n}{p} \right) \exp(2 \pi i n/p)
=\sum_{m=0}^{p-1} \exp ( 2 \pi i m^2/p)
=\sqrt{(-1)^{\frac{p-1}{2}} p}.
\label{3e24}  
\end{align}
\end{thm}

\begin{proof}[Proof of second expression]
We consider the quantity
\begin{align}
I
=
\sum_{\substack{a=1,\,a=\text{quadratic} \\ \text{residue}} }^{p-1}
\exp(2 \pi i a/p),
\label{3e25}
\end{align}
in $a \in \F$.
$\left\{ 1^2, 2^2, \cdots, n^2, \cdots, (p-n)^2, \cdots, (p-1)^2 \right\}$
are quadratic residues, and the same quadratic residue comes twice
as $n^2 \equiv (p-n)^2\ ({\rm mod}\ p)$.
Then we have 
$\displaystyle{I=\frac{1}{2} \sum_{m=1}^{p-1} \exp( 2 \pi i m^2/p)}$.
While we obtain
\begin{align}
G_p
&=\sum_{m=1}^{p-1} \left( \frac{m}{p} \right) \exp( 2 \pi i m/p)
=\sum_{\substack{a=1,\,a=\text{quadratic}\\ \text{residue}} }    ^{p-1} \exp(2\pi i a/p)
-\sum_{\substack{a=1,\,a=\text{quadratic}\\ \text{non-residue}} }^{p-1} \exp(2\pi i a/p)
\nonumber 
\\
&=I-J.
\label{3e26}
\end{align}
By the way, we obtain
$\displaystyle{I+J=\sum_{m=1}^{p-1} \exp( 2 \pi i m/p)=-1}$, which gives $J=-I-1$. Then we conclude 
\begin{align}
G_p=2 I+1=1+\sum_{m=1}^{p-1} \exp( 2 \pi i m^2/p)=\sum_{m=0}^{p-1} \exp( 2 \pi i m^2/p).
\label{3e27}
\end{align}
\end{proof}
Thus we obtain the generalized Gaussian sum with the other expression in the form 
\begin{align}
G(\tau)
=\sum_{n=1}^{\infty} \left(\frac{n}{p} \right) \exp(2 \pi i n \tau)
=\sum_{m=0}^{\infty} \exp ( 2 \pi i m^2 \tau).
\label{3e28}  
\end{align}
Using this generalized Gaussian sum with the other expression, 
we can connect  the generalized Gaussian sum with the elliptic theta 
function in the form
\begin{align}
G(\tau)
=1+\sum_{n=1}^{\infty} \exp( 2 \pi i n^2 \tau)
=\frac{1}{2} \left( 
   \vartheta\left[\begin{array}{@{\,}c@{\,}}
                    1 \\
                    1 \\ 
                  \end{array}\right]
  (0,\tau)
  +1 \right).
\label{3e29}
\end{align}
The elliptic theta function has the structure of the modular transformation, 
and by shifting the constant value, the generalized Gaussian sum 
also has the structure of the modular transformation.

Through the considerations above, we conclude the generalized Gaussian sum is the quadratic curve analogue of the modular form in the Taniyama-Shimura conjecture.

\section{%
Summary and Discussions}

We have examined the quadratic curve analogue of the Taniyama-Shimura conjecture for
the quadratic curves. The number of solutions in $\FF$ is governed by the order of
the quadratic curve analogue of the Mordell-Weil group.
For quadratic curves $y^2 \equiv a x^2+1\ ({\rm mod}\ p)$,
the order of the group of the Mordell-Weil analogue is given by
$\displaystyle{N(p)=p-\left( \frac{a}{p} \right)}$.
If we make the combination of 
$\displaystyle{b(p)=p-N(p)=\left( \frac{a}{p} \right)}$, we obtain 
\begin{align*}
-b(p)=N(p)-p=\sum_{n=1}^{p-1} \left( \frac{ a x^2+1}{p} \right).
\label{5e1}
\end{align*}

For the quadratic curve analogue of the modular form of the Taniyama-Shimura conjecture, 
by replacing the Legendre symbol with the Kronecker symbol, 
we obtain the generalized Gaussian sum. If we use the other form of the Gaussian sum,
the generalized Gaussian sum is connected with the elliptic theta function with zero 
argument. Thus, the generalized Gaussian sum has the structure of the 
modular transformation.

The quadratic curve analogue of the conductor of the elliptic curve over $\FF$ 
 is the discriminant 
of quadratic curves $y^2 \equiv a x^2+1 \pmod p$,
which is given by 
$$
N_C
=
\left\{
 \begin{array} {rl}
   a, & \text{if~}a \equiv 1    \pmod 4,\\
  4a, & \text{if~}a \equiv 2, 3 \pmod 4.
 \end{array} \right.
$$
The meaning of this conductor $N_C$  is given as follows:
\begin{enumerate}
\item For the prime $p$ with 
$(p,4a) \ne 1, ~\text{the quadratic curves}~ y^2 \equiv a x^2+1 \pmod p$ 
reduces to the equation of points $y^2 \equiv 1 \pmod p$.
\item For the prime $p$ with $(p,4a) = 1$,
$b(p)=\left( \dfrac{a}{p} \right)_K$
has the periodic property of the form $b(p+N_C)=b(p)$.
\end{enumerate}

The quadratic curve analogue of the level of the congruent modular form is the integer $N_L$ in such a way that the generalized Gaussian sum becomes periodic by setting
$q^{N_L}=1$ if $a \equiv 1 \pmod 4$ 
and 
$q_1^{N_L}=1$ if $a \equiv 3 \pmod 4$.
Therefore, we obtain 
$N_L=a $ if $a \equiv 1 \pmod 4$ and
$N_L=4a$ if $a \equiv 3 \pmod 4$.

Thus, for the quadratic curves, the conductor $N_C$ and $N_L$ takes the same value
\[
 N_C=N_L=
\left\{ \begin{array}{rl} 
          a, & \text{if~}a \equiv 1 \pmod 4,\\
         4a, & \text{if~}a \equiv 3 \pmod 4.\ 
        \end{array}
\right. 
\]

\begin{appendices}
\setcounter{equation}{0}
\section{Correspondence between \\
$\bm{n~\in\left\{1, 3, 5, \cdots \widecheck{a}, \cdots, 2 a-1\right\}}$ 
and $\bm{m~\in\left\{ 1, 2, 3, \cdots , a-1 \right\}}$ for odd $\bm a$}
\setcounter{equation}{0}

In this appendix, we show that
$n~\in\left\{1, 3, 5, \cdots \widecheck{a}, \cdots, 2 a-1\right\}$
and $m~\in\left\{ 1, 2, 3, \cdots , a-1 \right\}$
are mapped to each other by using the following relation with a suitable integer $\ell$ for odd $a$:
\begin{align}
  n=4m-a(2\ell-1), \quad \ell=0,~1,~2,\cdots.
\label{eq:relation_nm}
\end{align}

\begin{thm}
If $m_1 \ne m_2$, then $n_1 \ne n_2$.
\end{thm}

\begin{proof}
Let 
$n_1=4m_1-a(2\ell_1-1)$ and 
$n_2=4m_2-a(2\ell_2-1)$. Then
\[
  n_1-n_2 = 2\left(2(m_1-m_2)-a(\ell_1-\ell_2)\right).
\]
Suppose $n_1=n_2$, then $2(m_1-m_2)=a(\ell_1-\ell_2)$. Since $a$ is odd, $\ell_1-\ell_2$ must be even. If $m_1 \ne m_2$, $|a (\ell_1-\ell_2)|\ge 2a $, while $|2(m_1-m_2)|\le 2(a-1)$ because $m_1$ and $m_2$ are elements of 
$\left\{ 1, 2, 3, \cdots , a-1 \right\}$. Then,  $2(m_1-m_2)-a (\ell_1-\ell_2)$ cannot be reduced to 0 for $m_1 \ne m_2$,
which contradicts the assumption $n_1=n_2$. Namely, we conclude that $n_1 \ne n_2$ if $m_1 \ne m_2$.
\end{proof}

\begin{thm}
If odd integer $n~\in\left\{1, 3, 5, \cdots \widecheck{a}, \cdots, 2 a-1\right\}$
 is given, then $m~\in\left\{ 1, 2, 3, \cdots , a-1 \right\}$ and $\ell$ are 
 determined from Eq.\eqref{eq:relation_nm}.
\end{thm}

\begin{proof}
To solve Eq.\eqref{eq:relation_nm}, we first consider the equation
$$1=4 X+ a Y.$$
Thanks to $(4,a)=1$, this equation always has an integer solution $X=X_0$ and $Y=Y_0$, where $Y_0$ must be odd because $a$ is odd. By multiplying $n$, we obtain the solution of Eq.\eqref{eq:relation_nm} as $m=m_0=n X_0$ and 
$-2(\ell-1)=-2(\ell_0-1)=n Y_0$, namely
\[
  n=4 m_0 -a(2\ell_0-1).
\]

It can be seen that Eq.\eqref{eq:relation_nm} has an infinite number of solutions. Indeed, for any integer $k$
\begin{align}
  m=m_0 + ka, \quad \ell=\ell_0 + 2ka,
\label{eq:m_and_n}
\end{align}
are solution of Eq.\eqref{eq:relation_nm}. By making $k$ an appropriate integer, $m$ can be an element of $\left\{ 1, 2, 3, \cdots , a-1 \right\}$ except when $m$ is 0. If $m=0$, however, $n=a$ with $\ell=0$. it is the excluded value for $n$.
The exclusion of $n=a$ follows from the fact that the definition of $\bar{f}(\tau)$ given in Eqs.\eqref{3e18} and \eqref{3e21} includes $\displaystyle{ \left( \frac{\mathstrut a}{n} \right)_K}$, which is 0 for $n=a$.
\end{proof}

\end{appendices}


\end{document}